\documentclass[11pt,letterpaper]{article}
\usepackage{amsmath,amssymb,amsthm}
\numberwithin{equation}{section}
\usepackage{MnSymbol} 
\usepackage{bm, bbm} %
\usepackage{graphicx,color}
\usepackage[hyphens]{url}
\usepackage{dsfont}
\usepackage{booktabs, makecell, adjustbox}
\usepackage[normalem]{ulem}
\usepackage{mathtools}
\usepackage{thmtools}
\usepackage[hidelinks,colorlinks=true,linkcolor=blue,citecolor=blue]{hyperref}
\usepackage[nameinlink,capitalize]{cleveref}
\usepackage{multirow}
\usepackage{algorithm}
\usepackage[noend]{algpseudocode}
\usepackage{tikz} %
\usepackage{pifont}
\usepackage{physics}
\usepackage{enumitem}

\usepackage{enumitem}

\usepackage{soul} %
\usepackage[square,sort,numbers]{natbib}

\usepackage{subcaption}

\usepackage{tabularx}
\newcolumntype{L}{>{\raggedright\arraybackslash}X}

\newtheorem{theorem}{Theorem}[section]

\newtheorem{lemma}[theorem]{Lemma}
\newtheorem{remark}[theorem]{Remark}

\newcommand{\R}{\mathbb{R}}

\renewcommand{\epsilon}{\varepsilon}

\crefname{section}{\mbox{\S\!\!}}{\mbox{\S\!\!}}

\usepackage[margin=1in]{geometry}

\makeatletter
\def\blfootnote{\gdef\@thefnmark{}\@footnotetext}
\makeatother

\title{Acyclic Monotone Operators Are Not Closed Under Addition}

	 \author{
	 	Henry Shugart
	 	\\	UPenn \\	\texttt{hshugart@upenn.edu}
	 }

\begin{document}

\maketitle

\begin{abstract}
    Borwein and Wiersma [SIAM J. Optim. 18(3) (2007), 946–960]
    asked if the set of acyclic monotone operators is closed under addition. We answer this question in the negative.
\end{abstract}

\section{Introduction}\label{sec:intro}

Recall that a single-valued operator $F$ is \textit{monotone} if $\langle F(z) - F(z'), z -z '\rangle \geq 0$ for every $z,z'$ in the domain of $F$. Further, a monotone operator $F$ is \emph{acyclic} if, for every representation $F = \partial g + S$, with $g$ convex and $S$ monotone, $\partial g$ is constant (i.e., $g$ is affine). Acyclicity plays a fundamental role in the study of monotone operators due to a classical result of~\citet{asplund1970monotone}, which states that every single-valued maximally monotone operator admits a decomposition
\begin{equation*}
F = \nabla g + S,
\end{equation*}
where $g$ is convex, $S$ is an acyclic monotone operator, and this decomposition is unique up to an additive constant in $\nabla g$ and $S$. This has motivated the study of acyclic monotone operators.

A question raised by~\citet[p.\,960]{borwein2007asplund} is whether the set of acyclic monotone operators is closed under addition. This holds true if (at least) one of the two summed acyclic monotone operators is linear \citep{borwein2007asplund}. However, in this note, we answer this question in the negative in general. We do this by constructing two (nonlinear) acyclic monotone operators whose sum is the gradient of a nonaffine convex function. Because the set of acyclic monotone operators is closed under nonnegative scalings, this result can be equivalently interpreted as showing that the set of acyclic monotone operators is not convex.

\section{Preliminaries}\label{sec:prelim}

Here we briefly recall three helpful facts about monotone operators that we use below. Proofs of these facts can be respectively found, e.g., in~\citep[Theorem 1]{Rockafellar_1970} ,\citep[Proposition 17.7]{bauschke2017convex}, and~\citep[Proposition 1.3]{bianchi2003pseudomonotone}. For the first fact, recall that a monotone operator is said to be maximal if it has no proper monotone extension; see, e.g., the textbooks~\citep{rockafellar1997convex,bauschke2017convex} for further background.

\begin{lemma}[Saddle-operators of convex-concave functions are monotone]\label{lem:saddle-operator}
If $f(x,y)$ is convex-concave, then the saddle-operator $(\partial_x f, -\partial_y f)$ is maximally monotone.
\end{lemma}

\begin{lemma}[Monotone gradient implies convex function]\label{lem:monotoneconvex}
A function $f\in C^1$ is convex if and only if its gradient $\nabla f$ is monotone.
\end{lemma}

\begin{lemma}[Bimonotone implies skew]\label{lem:skew}
Let $F:\R^d\rightarrow \R^d$. If $F$ and $-F$ are both monotone,
then there exist a skew-symmetric linear map $\bm A$ and vector $b$ such that $F(z) = \bm A z + b$.
\end{lemma}

\section{Main Result}\label{sec:main}
Consider the functions 
\begin{align*}
f_+ (x,y) = \frac{1}{2}x^2 +\sin x \sin y -\frac{1}{2} y^2,  \qquad f_-(x,y) = \frac{1}{2}x^2 -\sin x \sin y -\frac{1}{2}y^2.
\end{align*}
Since both functions are convex-concave\footnote{This can be checked by examining second derivatives. For example, $\partial^2 f_{+} / \partial x^2 = 1 - \sin x \sin y \geq 0$ and $\partial^2 f_{+} / \partial y^2 = -1 - \sin x \sin y \leq 0$, and similarly for $f_-$.}, the corresponding saddle operators
$$F_+(x,y) = (x +\cos x  \sin y , y - \sin x \cos y   ),\ F_-(x,y) = (x - \cos x  \sin y , y + \sin x \cos y    )$$
are maximally monotone by~\cref{lem:saddle-operator}.
 Denote the Asplund decompositions of $F_+, F_-$ by
\begin{equation}\label{eq:counterexampleasplund}
    F_+ = \nabla g_+ + S_+, \qquad F_- = \nabla g_- + S_-.
\end{equation}

\begin{theorem}[Main result]
The class of acyclic monotone operators is not closed under addition. Specifically, $S_+$, and $S_-$, as defined in~\eqref{eq:counterexampleasplund}, are both acyclic monotone operators, yet $S_+ + S_-$ is not acyclic.
\end{theorem}
\begin{proof}
By construction, $S_+$ and $S_-$ are the acyclic parts of the decompositions in~\eqref{eq:counterexampleasplund}. Below we show that $S_+ + S_-$ admits the decomposition $S_+ + S_- = \nabla \Phi +0 $ for some non-affine convex function $\Phi$; by definition this implies that $S_+ + S_-$ is not acyclic.

We first show that $S_+ + S_-$ is the gradient of a convex function. Adding the two decompositions in~\eqref{eq:counterexampleasplund} gives
\begin{equation*}
(2x,2y) = F_+(x,y) + F_-(x,y) = \nabla g_+ + \nabla g_- + S_+ + S_-.
\end{equation*}
The map $(2x,2y)$ is the gradient of $h(x,y)=x^2+y^2$, so $S_+ + S_- = \nabla \Phi$ is the gradient of a function $\Phi = h - g_+ - g_-$.  Moreover, since $\nabla \Phi = S_+ + S_-$ is the sum of two monotone operators, it is monotone, hence $\Phi$ is convex by~\cref{lem:monotoneconvex}.

It remains to show that $\Phi$ is not affine. Assume for the sake of contradiction that it is, i.e., that $\nabla \Phi=c$ for some constant $c\in \R^2$. 
Because $S_-$ is monotone and $S_- = c-S_+$, it follows that $-S_+$ is monotone (adding a constant does not affect monotonicity). Since $S_+$ and $-S_+$ are both monotone,~\cref{lem:skew} implies that $S_+$ is skew-affine, i.e., that $S_+(x,y) = (ay+b_1,\,-ax+b_2)$
 for some constants $a,b_1,b_2 \in \R$. It follows that $\nabla g_+ = F_+ - S_+ = (x+\cos x  \sin y  - ay-b_1, y - \sin x \cos y   +ax-b_2)$. Note that the mixed partial derivatives
$$\frac{\partial^2}{\partial y \partial x}g_+ = -a+\cos x  \cos y   , \qquad \frac{\partial^2}{\partial x \partial y} g_+=a-\cos x  \cos y   ,$$
are not identical for any choice of $a$. This implies that $\nabla g_+$ is not integrable, yielding a contradiction.

\end{proof}

\begin{remark}
    The counterexample is not unique. The same argument applies if $\sin x \sin y$ is replaced by a different function $u(x,y)\in C^2$ such that $ |\partial^2 u / \partial x^2 |, |\partial^2 u / \partial y^2|\leq 1$, and $\partial^2 u / \partial x \partial y$ is nonconstant.
\end{remark}

\begin{remark}
    This specific operator $F_+$ was studied in the simultaneous, independent work of \citet{lee2026nesterov}. The focus of their paper is to develop faster algorithms for solving variational inequalities for monotone operators with known Asplund decompositions. In Section 2 they derive an explicit Asplund decomposition of $F_+$ as it provides a natural example of a monotone operator that has a nonlinear acyclic component. In our paper, we study $F_+$ for the different purpose of answering the aforementioned question of~\citet{borwein2007asplund}.
\end{remark}

\paragraph*{Acknowledgements.} The author is grateful to Jason Altschuler for helpful comments and suggestions.
\small
\bibliographystyle{plainnat}
\bibliography{refs}
\normalsize

\end{document}